\documentclass{article}
\usepackage{amsmath}
  \usepackage{paralist}
  \usepackage{graphics} %% add this and next lines if pictures should be in esp format
  \usepackage{epsfig} %For pictures: screened artwork should be set up with an 85 or 100 line screen
\usepackage{graphicx}  \usepackage{epstopdf}%This is to transfer .eps figure to .pdf figure; please compile your paper using PDFLeTex or PDFTeXify.
 \usepackage[colorlinks=true]{hyperref}
\hypersetup{urlcolor=blue, citecolor=red}

\usepackage{amsthm}
\newtheorem{theorem}{Theorem}[section]
\newtheorem{corollary}[theorem]{Corollary}
\newtheorem{lemma}[theorem]{Lemma}
\newtheorem{proposition}[theorem]{Proposition}

\newtheorem{definition}[theorem]{Definition}
\newtheorem{remark}[theorem]{Remark}

\usepackage{amssymb}
\newcommand{\Int}[0]{\mathrm{Int } }

\newcommand{\Unstab}[1]{\ensuremath{W_\varepsilon^{(u)}(#1)}}

\title{Differentiability of Hausdorff dimension of the non-wandering set in a planar open billiard}

\author{Paul Wright}

\begin{document}
\maketitle

% Enter the first author's name and address:
%\centerline{\scshape Paul Wright}
%\medskip
%{\footnotesize
%% please put the address of the first author
% \centerline{Department of Mathematics and Statistics, The University of Western Australia}
%   \centerline{ Perth, Western Australia}
%} % Do not forget to end the {\footnotesize by the sign }

%\bigskip

% The name of the associate editor will be entered by an editorial staff
% "Communicated by the associate editor name" is not needed for special issue.
% \centerline{(Communicated by the associate editor name)}

%The abstract of your paper
\begin{abstract}
We consider open billiards in the plane satisfying the no-eclipse condition. We show that the points in the non-wandering set depend differentiably on deformations to the boundary of the billiard. We use Bowen's equation to estimate the Hausdorff dimension of the non-wandering set of the billiard. Finally we show that the Hausdorff dimension depends differentiably on sufficiently smooth deformations to the boundary of the billiard, and estimate the derivative with respect to such deformations. 
\end{abstract}

\section{Introduction} \label{S intro}

The dimension theory of dynamical systems studies the dimensional characteristics (such as Hausdorff dimension) of the invariant sets of dynamical systems. See \cite{Pesinbook} for an introduction to the theory, or \cite{BGsurvey} for a recent review of this field. Past work has examined how dimensional characteristics of various dynamical systems can change with respect to perturbations of the system; for example the differentiability of entropy of Anosov flows \cite{Katok}, SRB measures in hyperbolic flows \cite{Ruelle}, and Hausdorff dimension of horseshoes \cite{Mane}. However this kind of problem has not been considered in the context of open billiard systems. The Hausdorff dimension of the non-wandering set has been estimated for open billiards in the plane \cite{Kenny} and in higher dimensions \cite{Wright1}. In this paper we show that the Hausdorff dimension of the non-wandering set for an open billiard in the plane depends smoothly on perturbations to the boundary of the billiard. Specifically, if the boundary of the billiard is $\mathcal{C}^r$-smooth and depends $\mathcal{C}^{r'}$-smoothly on a perturbation parameter $\alpha$, then the Hausdorff dimension is $\mathcal{C}^{\min\{r-3, r'-1\}}$-smooth with respect to $\alpha$. Further, we find bounds for the derivative of the Hausdorff dimension with respect to $\alpha$, and we show that if the boundary is real analytic then the dimension is real analytic. 

A billiard is a dynamical system in which a single pointlike particle moves at constant speed in some domain $Q \subset \mathbb{R}^n$ and reflects off the boundary $\partial Q$ according to the classical laws of optics \cite{Sinai}. Open billiards are a class of billiard in which the domain $Q$ is unbounded. Let $K= K_1 \cup \ldots \cup K_m$ $(m \geq 3)$ be a subset of $\mathbb{R}^2$, where each $K_i$ is a compact strictly convex disjoint domain in $\mathbb{R}^2$ with $\mathcal{C}^r$ boundary ($r \geq 3$). Set $Q = \overline{\mathbb{R}^2 \backslash K}$. We assume that $K$ satisfies the \textit{no-eclipse condition} $(\textbf{H})$ introduced by Ikawa in \cite{Ikawa}:
\begin{description}
	\item[$(\textbf{H})$] For distinct $1 \leq i,j,k \leq m$, the convex hull of $K_i \cup K_j$ is disjoint from $K_k$.
\end{description}
This condition ensures that the collision angle $\phi$ is bounded above by a constant $\phi_{\max} < \frac{\pi}{2}$, and prevents discontinuities in the non-wandering set $M_0$ which consists of all bounded billiard trajectories in $Q$. 

In this paper, we consider smooth homotopies (with a parameter $\alpha$) between different billiards, which we call billiard deformations. Section \ref{billiards} contains some preliminaries on open billiards and a precise definition of the deformations. We show that the periodic trajectories in the non-wandering set are differentiable and Lipschitz with respect to $\alpha$. We extend this to the whole non-wandering set, and show that the curvature of the stable and unstable manifolds are also differentiable and Lipschitz. 

The Hausdorff dimension of the non-wandering set was estimated in \cite{Kenny} by investigating convex fronts. This was later improved and extended to higher dimensional billiards in \cite{Wright1}. In this paper, using techniques from Pesin's book on dimension theory in dynamical systems \cite{Pesinbook}, we recover the estimates in \cite{Kenny} and show that they also apply to the lower and upper box dimensions. That is,
$$\frac{2 \log(m-1)}{\log(1 + d_{\max} k_{\max})} \leq \underline{\dim}_B M_0 = \overline{\dim}_B M_0= \dim_H M_0 \leq \frac{2 \log(m-1)}{\log(1 + d_{\min} k_{\min})},$$
where $d_{\min}, d_{\max}, k_{\min}, k_{\max}$ are constants that depend on simple geometric characteristics of the obstacles.
Furthermore we show that these dimensions depend differentiably on the boundary of the billiard obstacles. That is, if the billiard is shifted or deformed smoothly with some parameter $\alpha$, then the function 
$$\mathcal{D}(\alpha) = \underline{\dim}_B M_0 = \overline{\dim}_B M_0= \dim_H M_0$$
is differentiable with respect to $\alpha$. In fact, it is almost as smooth as the deformation, and for the first derivative we have
$$\left| \frac{d}{d\alpha} \mathcal{D}(\alpha) \right| \leq C,$$
where $C$ is a constant depending only on simple geometrical characteristics of the obstacles. %Similar results have been achieved for other dynamical systems (see e.g. horseshoes in \cite{Mane}), but this is the first time it has been done for non-wandering sets of open billiards. 

\section{Open billiards} \label{billiards}
Consider the set $Q$ described in the introduction. We describe a particle in the billiard by $S_t x = x_t = (q_t, v_t)$ where $q_t \in Q$ is the position of the particle and $v_t \in \mathbb{S}^1$ is its velocity at time $t$. The map $S_t$ is called the \textit{billiard flow}. Then for as long as the particle stays inside $Q$, it satisfies 
$$(q_{t+s}, v_{t+s}) = S_s(x_t)= (q_t + s v_t, v_t).$$
\noindent Collisions with the boundary are described by 
$$v^+ = v^- - 2 \langle v^-, n \rangle n,$$
where $n$ is the normal vector (into $Q$) of $\partial Q$ at the point of collision, $v^-$ is the velocity before reflection, and $v^+$ is the velocity after reflection.

\subsection{Non-wandering set.}
For $x = (q,v)$ with $q \in \partial K$, $v \in \mathbb{S}^1$, we denote by $n = n_K(q)$ the outward unit normal vector of $\partial K$ at $q$, by $\phi(x)$ the angle between $v$ and $n$, by $\phi_{\max}$ the supremum of this angle over $M_0$, by $\kappa(q)$ the curvature of $\partial K$ at $q$, and by $t_j(x) \in [-\infty, \infty]$ the time of the $j$-th reflection of $x$ (with the convention that $t_0(q,v) = t_1(q, -v)$ if $q \in \Int(Q)$, or $t_0(q, v) = 0$ if $q \in \partial Q$). If the forward trajectory of $x$ does not have at least $j$ reflections, then $t_j(x) = \infty$, and if the backward trajectory does not have at least $j$ reflections then $t_{-j}(x) = -\infty$. Let $d_j(x) = t_j(x) - t_{j-1}(x)$. Let $M = \{(q,v) \in \partial K \times \mathbb{S}^1:  \langle n(q), v \rangle \geq \cos \phi_{\max} \}$, and $M' = \{x \in M: t_1(x) < \infty\}$. Let $\pi: M \rightarrow \partial K$ be the canonical projection $(q,v) \mapsto q$. Then define the \textit{billiard map} $B: M' \rightarrow M$ by $Bx = S_{t_1(x)}(x)$. Then $B$ is $\mathcal{C}^{r-1}$. On the tangent space $T_x M$, for $x \in M$, we will use the norm $\|(dq, dv)\| = \cos \phi |dq|$ (see e.g. \cite{ChaoticBilliards}). 

The set $M$ together with an inner product inducing this norm is a Riemannian manifold. The \textit{non-wandering set} of a billiard is the set of points whose trajectories are bounded. The non-wandering set of the billiard flow is denoted $\Omega(S)$ or $\Omega$, and its restriction to the boundary of $K$ is $M_0 = \Omega \cap (\partial K \times \mathbb{S}^1)$. Equivalently, $M_0 = \{x \in M: |t_j(x)| < \infty, \forall j \in \mathbb{Z} \}$ is the non-wandering set of the billiard map $B$. Then $B$ is a $\mathcal{C}^{r-1}$ diffeomorphism on $M_0$.

\subsection{Notation for upper bounds on derivatives} \label{C_f notation}
We will say that a function of two variables $f(x,y)$ is called $\mathcal{C}^{A,B}$\textit{-smooth} or simply $\mathcal{C}^{A,B}$ if for every $a \leq A, b \leq B$ the derivatives $\frac{\partial^{a+b}f}{\partial x^a \partial y^b}$ are continuous. Frequently we will have some quantity that depends on a scalar $\alpha$ and a vector (or sometimes a scalar) $u$, and we will show that its derivatives are bounded by some constants. Rather than numbering these constants, we will label them with the quantity being differentiated in the subscript and the number of differentiations in the superscript. So for example if $f$ is a function of $\alpha$, we will say $\displaystyle \left| \frac{d^2 f}{d \alpha^2} \right| \leq C_f^{(2)}$. If $g$ is a function of $u = (u_0, \ldots, u_{n-1})$ and $\alpha$, then for each $q,q' \geq 0$ (but not $q = q' = 0$) we will say 
$\displaystyle \left| \frac{\partial^{q'} }{\partial \alpha^{q'}} \nabla^q g \right| \leq C_g^{(q,q')}$, for all $u$ in its domain. These constants may depend on $\alpha$, but not on $u$. When there is only one variable and only the first derivative is required, we will simply write $\displaystyle \left| \frac{d f}{d \alpha} \right| \leq C_f$. It will be clear what each constant refers to each time we use this notation.

\subsection{Billiard deformations.} \label{section 2D billiard deformations}
Here we define precisely what we mean by deformations to the boundary. We will always assume the boundaries $\partial K_i$ of each obstacle are parametrised counterclockwise. 

Let $I \subseteq [-\infty, \infty]$ be a closed interval. A deformation will be described by adding an extra variable $\alpha \in I$ to the parametrisations $\varphi_i$, so that any point on $\partial K_i$ is described by $\tilde{\varphi}_i(\tilde{u}_i, \alpha)$. Denote the perimeter of $\partial K_i(\alpha)$ by $L_i(\alpha)$. Then let 
$$R_i = \{(\tilde{u}_i, \alpha): \alpha \in I, \tilde{u}_i \in [0, L_i(\alpha)]\}.$$

\begin{definition} \label{deformation}
	Let $I \subseteq [-\infty, \infty]$ be a closed interval and let $m \geq 3$ be an integer. For any $\alpha \in I$, let $K(\alpha)$ be a subset of $\mathbb{R}^2$. For integers $r \geq 2, r' \geq 1$, we call $K(\alpha)$ a $\mathcal{C}^{r, r'}$-\textit{billiard deformation} if the following conditions hold for all $\alpha \in I$:
	\begin{enumerate}
		\item $K(\alpha) = \displaystyle\bigcup_{i=1}^m K_i(\alpha)$ satisfies the no-eclipse condition $(\textbf{H})$.
		\item Each $K_i(\alpha)$ is a compact, strictly convex set with $\mathcal{C}^r$ boundary and total arc length $L_i(\alpha)$.
		\item Each $K_i$ is parametrized counterclockwise by arclength with $\mathcal{C}^{r,r'}$ functions $\tilde{\varphi}_i: R_i \rightarrow \mathbb{R}^2$.
		\item For all integers $0 \leq q \leq r$, $0 \leq q' \leq r'$ (apart from $q = q' = 0$), there exist constants $C_\varphi^{(q,q')}$ depending only on $\alpha$ and the parametrizations, such that for all integers $i = 1, \ldots, m$, 
		$$\left\|\frac{\partial^{q + q'} \tilde{\varphi}_i}{\partial \tilde{u}_i^q \partial \alpha^{q'}} \right\| \leq C^{(q,q')}_{\varphi}.$$
	\end{enumerate}
\end{definition}

We call $\alpha$ the \textit{deformation parameter}, and the $C_\varphi^{(q,q')}$ \textit{deformation constants}. We assume that only one obstacle is affected by the deformation. This results in stronger estimates for the derivatives. The general case can be covered by considering several successive deformations, or by deforming several at once (see Remark \ref{rem multiple obstacles} for details on this). 
Define a function $\delta_i$ such that $\delta_i = 0$ if $K_i(\alpha) = K_i$ is constant for all $\alpha$, and $\delta_i = 1$ if $K_i(\alpha)$ depends on $\alpha$.

Since the obstacles are parametrized by arclength, we always have $\mathcal{C}^{(1,0)} = 1$, in fact $\left|\frac{\partial \tilde{\varphi}_i}{\partial \tilde{u}_i} \right| = 1$. The curvature of $\partial K(\alpha)$ is $\tilde{\kappa}_i(\tilde{u}_i, \alpha) = \left\langle n_i, \frac{\partial^2 \tilde{\varphi}_i}{\partial \tilde{u}_i^2}\right \rangle$, which is bounded below by $\kappa_{\min}$ and above by $\kappa_{\max} = C_\varphi^{(2,0)}$. 

\subsection{Shift maps and billiard expansions}
A billiard deformation is called a \textit{shift map} if $K_i(\alpha) = K_i(0) + \alpha v$ for some constant vector $v$. If the deformation is a shift map we can use the parametrization $\tilde{\varphi}_i(u, \alpha) = \hat{\varphi}_i(u) + \alpha v$ for some function $\hat{\varphi}$ that parametrises $\partial K_i(0)$. A shift map satisfies $C_\varphi^{(q,1)} = 0$ for all $q \geq 1$, and $C_\varphi^{(0,1)} = \|v\|$. 

For a given billiard $\hat{K} = \hat{K}_1 \cup \ldots \cup \hat{K}_m$, fix a point $r_i \in \hat{K}_i$ for each $i = 1, \ldots, m$. We call the deformation a \textit{billiard expansion} if $K_i(\alpha) = \hat{K}_i + \alpha r_i$ for every $i$. The effect of this map is to move all the obstacles apart without changing their shape. These maps are considered in the limit $\alpha \rightarrow \infty$ in \cite{Kenny}.

\subsection{Symbolic model.} \label{Symbolic model}
Let 
$$\Sigma_n = \{ \xi = (\xi_1, \ldots, \xi_n): \xi_i \in \{1, \ldots m\}, \xi_i \neq \xi_{i+1}, \xi_n \neq \xi_1 \}.$$
This is the symbol space that models $n$-periodic trajectories, that is trajectories $x$ such that $B^n x = x$. Let $M_n \subset M_0$ be the set of $n$-periodic trajectories. Let $(q_j, v_j) = B^jx$ and let $u_j \in [0, l_j]$ such that $\varphi_{\xi_j}(u_j) = q_j$. 

Let the \textit{two-sided subshift} $\sigma : \Sigma \rightarrow \Sigma$ be defined by $(\sigma \xi)_i = \xi_{i+1}$. Then $\sigma$ is continuous under the following metric $d_\theta$ for any $\theta \in (0,1)$.
\begin{displaymath}
d_\theta(\xi,\xi') = \begin{cases}
0:        &  \text{if $\xi_i = \xi'_i$ for all $i \in \mathbb{Z}$}  \\
\theta^n: &  \text{if $n = \max \{j \geq 0: \xi_i = \xi'_i \mbox{ for all } |i < j\}$},
\end{cases}
\end{displaymath}

For any point $x \in M_n$, define the corresponding sequence $\xi = (\xi_1, \ldots, \xi_n) \in \Sigma_n$ such that $\pi B^j x \in K_{\xi_j}$ for all $j = 1, \ldots, n$. We denote $K_\xi = K_{\xi_1} \times \ldots \times K_{\xi_n}$. Let the \textit{length function} $F = F_\xi: K_\xi \rightarrow \mathbb{R}$ be defined by 
$$F(q_1, \ldots, q_n) = \displaystyle \sum_{j=1}^n \|q_j - q_{j+1}\|,$$
where we write $q_{n+1} = q_1$. Consider the function $G = G_\xi: [0, l_{\xi_1}] \times \ldots [0, l_{\xi_n}] \rightarrow \mathbb{R}$ defined by $G(u_1, \ldots u_n) = F(\varphi_{\xi_1}(u_1), \ldots \varphi_{\xi_n}(u_n))$. If $K(\alpha)$ is a billiard deformation, then $G$ also depends on $\alpha$ and is $\mathcal{C}^{r, r'}$-smooth.

\begin{lemma} (\cite{Stoyanov1}, see also \cite{Stoyanovbook}) \label{pp lemma}
	If $K(\alpha)$ is a billiard deformation, then for a fixed $\xi$ the function $G_\xi$ has exactly one minimum at 
	$$u(\alpha) = (u_1(\alpha), \ldots, u_n(\alpha)).$$
	$F_\xi$ has a corresponding minimum 
	$$(p_1, \ldots, p_n) = (\varphi_{\xi_1}(u_1(\alpha), \alpha), \ldots, \varphi_{\xi_n}(u_n(\alpha), \alpha)).$$
	These points determine a billiard trajectory that satisfies the classical laws of optics.
\end{lemma}
This shows that the map $x \mapsto \xi$ is invertible and its inverse is $\chi \xi = (p_1, v_{12})$, where $v_{12}$ is the unit vector from $p_1$ to $p_2$, and the $p_i$ are found by minimizing the length function. For any $\theta \in (0, 1)$, $\chi$ is a homeomorphism from $M_0$ onto $(\Sigma, d_\theta)$, and the shift map $\sigma$ is topologically conjugate to $B$, that is $B = \chi^{-1} \circ \sigma \circ \chi$ (see e.g. \cite{Morita, Stoyanov2}).

For any $x$, let $\kappa_j = \kappa(\pi B^j x)$ be the curvature at $\pi B^j x$, $\phi_j = \phi(B^j x)$ be the angle between the velocity vector and the normal vector of $B^j x$, and let $\gamma_j = \frac{2 \kappa_j}{\cos \phi_j}$. Let $d_{\min}$, $\kappa_{\min}$ and $\gamma_{\min}$ be the minimum values of $d_1(x)$, $\kappa_0(x)$ and $\gamma_0(x)$ respectively over all $x \in M_0$, and let $d_{\max}, \kappa_{\max}$ and $\gamma_{\max}$ be the respective maximum values. Note that $\gamma_{\min} = 2 \kappa_{\min}$ and $\gamma_{\max} = \frac{2 \kappa_{\max}}{\cos \phi_{\max}}$. Also recall that $\phi_{\max} < \frac{\pi}{2}$ is the maximum value of $\phi(x)$ over $x \in M_0$. Whenever we are considering a fixed sequence $\xi$ we will use the abbreviation $\varphi_j = \varphi_{\xi_j}$.

\section{Derivatives of parameters} \label{Derivatives of parameters}
Let $K(\alpha)$ be a $\mathcal{C}^{r, r'}$ billiard deformation satisfying the conditions in Definition \ref{deformation}. Fix a finite admissible sequence $\xi = (\xi_0, \ldots, \xi_{n-1}) \in \Sigma_n$. Let 
$$
R_\xi	= \{(u, \alpha): \alpha \in I, u = (u_0, \ldots, u_{n-1}), u_j \in [0, L_{\xi_j}(\alpha)] \mbox{ for } j = 0, \ldots, n-1\}.
$$
For each $j = 0, \ldots, n-1$ set $\varphi_j = \tilde{\varphi}_{\xi_j}$. By Lemma \ref{pp lemma}, there exist numbers $u_j(\alpha) = u_j(\xi, \alpha)$ and points $p_j(\alpha) = p_j(\xi, \alpha) = \varphi_j(u_j(\alpha), \alpha) \in \partial K_{\xi_j}$ which correspond to a billiard trajectory. 

\begin{theorem}\label{thm 6 diff u}
	Let $K(\alpha)$ be a $\mathcal{C}^{r, r'}$ billiard deformation, with $r \geq 2, r' \geq 1$. For any finite admissible sequence $\xi \in \Sigma_n$, let $p_j = \varphi_j(u_j(\alpha), \alpha)$ be the periodic points corresponding to $\xi$. Then the parameters $u_j(\alpha)$ are $\mathcal{C}^{\min\{r -1, r' \}}$ with respect to the deformation parameter $\alpha$.
\end{theorem}

\begin{proof} Fix a sequence $\xi$ with period $n$. Recall that the periodic points corresponding to $\xi$ are given by the global minimum of the length function $G = G_\xi: R_\xi \rightarrow \mathbb{R}$ defined by 
$$G(u, \alpha) = \sum_{j=1}^n \|\varphi_j(u_j, \alpha) - \varphi_{j-1}(u_{j-1}, \alpha)\|.$$
This is a $\mathcal{C}^{r, r'}$ function of $u$ and $\alpha$. We will use the notation $I_j = \{j-1, j+1\}$. For each $j$, we can take the partial derivative of $G$ with respect to $u_j$ to get the equation
$$\frac{\partial G}{\partial u_j}(u, \alpha) = \displaystyle \sum_{i \in I_j} \left \langle \frac{\varphi_j(u_j, \alpha) - \varphi_i(u_i, \alpha)}{\|\varphi_j(u_j, \alpha) - \varphi_i(u_i, \alpha)\|}, \frac{\partial \varphi_j}{\partial u_j}(u_j, \alpha) \right \rangle.$$
By Lemma \ref{pp lemma}, for each $\alpha \in I$ the function $G$ has a single critical point $u = (u_1, \ldots, u_n)$, which satisfies
$$\frac{\partial G}{\partial u_j}(u(\alpha), \alpha) = 0 \mbox{ for all } j = 0, \ldots, n-1.$$
Now define a function $h_j: R_j \times I \rightarrow \mathbb{R}^n$ by $h_j(u, \alpha) = \frac{\partial G}{\partial u_j}(u, \alpha)$, and let $h$ be the vector $(h_0,\ldots, h_{n-1})$. This is a $\mathcal{C}^{r-1, r'}$ function of $u$ and $\alpha$. The Jacobian of $h$ with respect to $u$ is the Hessian matrix of $G$:
$$H_{ij} = \frac{\partial^2 G}{\partial u_i \partial u_j}.$$
This matrix is invertible (see \cite{Stoyanov1}), so we can apply the implicit function theorem. There exists a function $u(\alpha)$ that satisfies $h(u(\alpha), \alpha) = 0$, and the $u_j(\alpha)$ are exactly the parameters that minimize $G$. So $\varphi_j(u_j(\alpha), \alpha)$, $j = 0, \ldots, n-1$ are the periodic points corresponding to $\xi$. Furthermore, by the implicit function theorem \cite{IFT}, $u_j(\alpha)$ is $\mathcal{C}^{\min\{r - 1, r'\}}$.
\end{proof}
\noindent By the implicit function theorem, we have the following system of equations:
$$\frac{\partial^2 G}{\partial \alpha \partial u_j}(u(\alpha), \alpha) + \displaystyle \sum_{i=1}^n \frac{\partial u_i}{\partial \alpha} \frac{\partial^2 G}{\partial u_i \partial u_j}(u(\alpha), \alpha) = 0,$$
which we can write as a matrix equation,
\begin{equation} \label{Hu = b}
H \frac{\partial u}{\partial \alpha} = - \frac{\partial}{\partial \alpha} \nabla G.
\end{equation}
The next step is to estimate the derivatives $\frac{\partial u_j}{\partial \alpha}$.
\begin{theorem} \label{thm 6 du bounded}
	For any $\xi \in \Sigma_n$, the derivatives of the parameters satisfy
	$$\left|\frac{\partial u_j}{\partial \alpha}\right| \leq \frac{1}{\cos \phi_j}\frac{C^{(0,1)}_\varphi + C^{(1,1)}_\varphi d_{\min}}{\kappa_{\min} d_{\min}}.$$
\end{theorem}

\begin{proof} 
The two following sections cover the proof of this theorem. 
We use the notation $a_{ij} = 1/\|p_i - p_j\|$, $v_{ij} = a_{ij}(p_i - p_j)$. Denote by $n_j$ the normal vector to $\partial K$ at $p_j$, by $\kappa_j$ the curvature at $p_j$, and by $\phi_j = \phi_j(p_j, v_{jj+1})$ the collision angle. We will use the following vector identity several times.
\begin{proposition} \label{vector identity prop}
	If $u, v, w$ are unit vectors in the plane, then 
	\begin{equation}
	\label{vector identity}
	\langle u, w \rangle - \langle v,u \rangle \langle v, w \rangle = \langle v, u^\perp \rangle \langle v, w^\perp \rangle,
	\end{equation}
	where $v^\perp$ is a positive (counterclockwise) rotation by a right angle.
\end{proposition}

\subsection{Estimating $-\frac{\partial}{\partial \alpha} \nabla G$}
\label{Estimating b}

Note that $v_{jj-1} + v_{jj+1} = -2 \cos \phi_j n_j$ where $\phi_j$ is the collision angle at $(p_j, v_{jj+1})$, and $n_j$ is the normal vector of $K_{\xi_j}$ at $\phi_j$. We also use the vector identity (\ref{vector identity}).

\begin{align*}
\frac{\partial^2 G}{\partial \alpha \partial u_j}	&= \frac{\partial}{\partial \alpha} \displaystyle \sum_{i \in I_j} \left \langle \frac{\varphi_j - \varphi_i}{\|\varphi_j - \varphi_i\|}, \frac{\partial \varphi_j}{\partial u_j} \right \rangle\\
&= \displaystyle \sum_{i \in I_j} \left \langle v_{ji}, \frac{\partial^2 \varphi_j}{\partial \alpha \partial u_j} \right \rangle + \displaystyle \sum_{i \in I_j} a_{ji}\left \langle \frac{\partial \varphi_j}{\partial \alpha}- \frac{\partial \varphi_i}{\partial \alpha}, \frac{\partial \varphi_j}{\partial u_j} \right \rangle\\
&- \displaystyle \sum_{i \in I_j} a_{ji}\left \langle v_{ji}, \frac{\partial \varphi_j}{\partial \alpha} - \frac{\partial \varphi_i}{\partial \alpha} \right \rangle \left \langle v_{ji}, \frac{\partial \varphi_j}{\partial u_j} \right \rangle\\
&= \displaystyle \sum_{i \in I_j} a_{ji}\left \langle v_{ji}, \frac{\partial \varphi_j}{\partial \alpha}^\perp - \frac{\partial \varphi_i}{\partial \alpha}^\perp \right \rangle \left \langle v_{ji}, \frac{\partial \varphi_j}{\partial u_j}^\perp \right \rangle - 2 \cos \phi_j \left \langle n_j, \frac{\partial^2 \varphi_j}{\partial \alpha \partial u_j} \right \rangle.
\end{align*}
We have $\left \langle v_{ji}, \frac{\partial \varphi_j}{\partial u_j}^\perp \right \rangle = \cos \phi_j$, $\left| \left \langle v_{ji}, \frac{\partial \varphi_j}{\partial \alpha}^\perp \right \rangle \right| \leq \delta_j C^{(0,1)}_\varphi$, and $\left| \left \langle n_j, \frac{\partial^2 \varphi_j}{\partial u_j \partial \alpha} \right \rangle \right| \leq \delta_j C^{(1,1)}_\varphi $. We get the inequality

\begin{equation} \label{6 b estimate}
\left| \frac{1}{\cos \phi_j} \frac{\partial^2 G}{\partial \alpha \partial u_j} \right| \leq C^{(0,1)}_\varphi \left(a_{jj-1} \delta_{j-1} + (a_{jj-1} + a_{jj+1}) \delta_j + a_{jj+1} \delta_{j+1} \right) + 2 \delta_j C^{(1,1)}_\varphi.
\end{equation}

Let $b_j = \frac{1}{\cos \phi_j} \frac{\partial^2 G}{\partial \alpha \partial u_j}$. Since only one obstacle is deformed, either $\delta_j = 0$ or both $\delta_{j-1}$ and $\delta_{j+1} = 0$, so let $b_{\max} = \frac{2 C^{(0,1)}_\varphi}{d_{\min}} + 2 C^{(1,1)}_\varphi$ and note that $|b_j| \leq b_{\max}$ for all $j$.

\subsection{The Hessian Matrix}
\label{Hessian}

The Hessian of $G$ is a matrix composed of the derivatives $\frac{\partial^2 G}{\partial u_j \partial u_i}$. This section follows \cite{Stoyanov1} and Section 2.2 of \cite{Stoyanovbook}. The first derivatives of $G$ can be written
$$\frac{\partial G}{\partial u_j}(u) = \displaystyle \sum_{i \in I_j} \left \langle \frac{\varphi_j - \varphi_i}{\|\varphi_j - \varphi_i\|}, \frac{\partial \varphi_j}{\partial u_j} \right \rangle.$$
If $i \in I_j$, we can use (\ref{vector identity}) to get

\begin{align*}
\frac{\partial^2 G}{\partial u_j \partial u_i}	&= - a_{ji} \left \langle \frac{\partial \varphi_j}{\partial u_j}, \frac{\partial \varphi_i}{\partial u_i} \right \rangle + a_{ji} \left \langle v_{ji}, \frac{\partial \varphi_j}{\partial u_j} \right \rangle \left \langle v_{ji}, \frac{\partial \varphi_i}{\partial u_i} \right \rangle\\
&= - a_{ji} \left \langle v_{ji}, \frac{\partial \varphi_j}{\partial u_j}^\perp \right \rangle \left \langle v_{ji}, \frac{\partial \varphi_i}{\partial u_i}^\perp \right \rangle\\
&= a_{ji} \cos \phi_j \cos \phi_i.
\end{align*}
Along the diagonal $i = j$ we have

$$\frac{\partial^2 G}{\partial u_j^2} = \sum_{i \in I_j} a_{ji} \left \langle \frac{\partial \varphi_j}{\partial u_j}, \frac{\partial \varphi_j}{\partial u_j} \right \rangle - \displaystyle \sum_{i \in I_j} a_{ji} \left \langle v_{ji}, \frac{\partial \varphi_j}{\partial u_j} \right \rangle^2 + \displaystyle \sum_{i \in I_j} \left \langle v_{ji}, \frac{\partial^2 \varphi_j}{\partial u_j^2} \right \rangle.$$
Recall that $v_{jj-1} + v_{jj+1} = -(2 \cos \phi_j) n_j$, where $n_j$ is the outward unit normal vector. Also recall that $\kappa_j = \left \langle n_j, \frac{\partial^2 \varphi_j}{\partial u_j^2} \right \rangle$.
So we have $\displaystyle \sum_{i \in I_j} \left \langle v_{ji}, \frac{\partial^2 \varphi_j}{\partial u_j^2} \right \rangle = 2 \kappa_j \cos \phi_j$. Using the vector identity (\ref{vector identity}) we get 

$$\frac{\partial^2 G}{\partial u_j^2} = (a_{jj-1} + a_{jj+1}) \cos^2 \phi_j + 2 \kappa_j \cos \phi_j.$$
Finally, if $i \notin I_j \cup \{j\}$, then $\frac{\partial^2 G}{\partial u_j \partial u_i} = 0$. We will now show the derivatives $\frac{\partial u_j}{\partial \alpha}$ are bounded.

\begin{proposition} \cite{Stoyanov1}
	The Hessian matrix $H$ is non-singular and positive definite.
\end{proposition}
\begin{proof}
A proof can be found in \cite{Stoyanov1} or \cite{Stoyanovbook}.
\end{proof}

\section{Solving the cyclic tridiagonal system} \label{4 solving system}

From (\ref{Hu = b}) and the results of Section \ref{Hessian}, we now have the following system of equations:
\begin{align*}
\frac{\partial^2 G}{\partial \alpha \partial u_j}		&=	- \sum_{i=1}^n \frac{\partial u_i}{\partial \alpha} \frac{\partial^2 G}{\partial u_i \partial u_j}\\
&=	a_{jj-1} \cos \phi_j \cos \phi_{j-1} \frac{\partial u_{j-1}}{\partial \alpha}\\
&+	\left((a_{jj-1} + a_{jj+1}) \cos^2 \phi_j + 2 \kappa_j \cos \phi_j \right) \frac{\partial u_j}{\partial \alpha}\\
&+	a_{jj+1} \cos \phi_j \cos \phi_{j+1} \frac{\partial u_{j+1}}{\partial \alpha}.
\end{align*}
For each $j$, make the substitutions $y_j = \frac{\partial u_j}{\partial \alpha} \cos \phi_j$, $\gamma_j = \frac{2 \kappa_j}{\cos \phi_j}$ and $a_j = a_{jj-1}$, $a_1 = a_{n1}$. Divide through by $\cos \phi_j$, then we can rearrange the system to 
$$\frac{1}{\cos \phi_j} \frac{\partial^2 G}{\partial \alpha \partial u_j} = a_j y_{j-1} + \left(a_j + a_{j+1} + \gamma_j \right) y_j + a_{j+1} y_{j+1}.$$
We can write this as a matrix equation $A y = b$, where $y = (y_1, \ldots, y_n)^\intercal$, $b = (b_1, \ldots, b_n)^\intercal$, and $A$ is a matrix. 
\[ \left( \begin{array}{cccccc}
a_1 + a_2 + \gamma_1	& a_2 					&			& 0 				& a_1\\
a_2 					& a_2 + a_3 + \gamma_2	& a_3		& 					& 0\\
& a_3					& \ddots	& \ddots 			&  \\
0			 			& 	 					& \ddots	& \ddots 			& a_n\\
a_1 					& 0 					& 			& a_n 				& a_n + a_1 + \gamma_n \end{array} \right) 
\left( \begin{array}{c} y_1\\ \vdots \\ \vdots \\ \vdots \\ y_n \end{array} \right) = \left( \begin{array}{c} b_1\\ \vdots \\ \vdots \\ \vdots \\ b_n \end{array} \right) .\]
A tridiagonal matrix only has non-zero elements in the main diagonal and the first diagonals above and below the main diagonal. $A$ is \textit{cyclic tridiagonal}, meaning it can have two more non-zero elements in the corners. It is also \textit{diagonally dominant by rows} since $a_j + a_{j+1} + \gamma_j > a_j + a_{j+1} > 0$. The problem now is to estimate the solutions $y_j$ of this equation. We could estimate 
$$\|y\|_2 \leq \frac{\|b\|_2}{\|A^{-1}\|^{-1}} \leq \sqrt{n} \frac{b_{\max}}{\|A^{-1}\|^{-1}}.$$
This may seem to be the obvious approach to take. However since $\sqrt{n}$ is unbounded we cannot use this to find constant bounds on $y_j$ that hold for all $n$. Instead we use the following theorem of Varah \cite{Varah}. Let $\| \hspace{2pt} \|_\infty$ denote the matrix norm induced by the infinity norm.

\begin{theorem} \cite{Varah}
	Let $A = (A_{ij})_{i,j = 1}^n$ be a diagonally dominant matrix. Then 
	$$\|A^{-1}\|_\infty \leq \frac{1}{h},$$
	where
	$$h = \min_i\left(|A_{ii}| - \sum_{j \neq i} |A_{ij}| \right).$$
\end{theorem}
\noindent For our matrix, we have 
$$h = \min_i\left((a_i + a_{i+1} + \gamma_i) - (a_i + a_{i+1}) \right) = \min_i \gamma_i \geq 2\kappa_{\min}.$$ Returning to the system $Ay = b$ we have $\|y\|_\infty \leq \|A^{-1}\|_\infty \|b\|_\infty$, so $|y_j| \leq \frac{b_{\max}}{2\kappa_{\min}}$. Recall that $b_{\max} = \frac{2 C^{(0,1)}_\varphi}{d_{\min}} + 2C^{(1,1)}_\varphi$ and $y_j = \frac{\partial u_j}{\partial \alpha} \cos \phi_j$. Then we get 
$$
\left|\frac{\partial u_j}{\partial \alpha}\right| \leq \frac{1}{\cos \phi_j} \frac{C^{(0,1)}_\varphi + C^{(1,1)}_\varphi d_{\min}}{\kappa_{\min} d_{\min}}.
$$

\end{proof}

\begin{corollary} \label{p cor}
	Recall that the periodic points $p_0, \ldots, p_{n-1}$ are given by \\ $p_j = \varphi_j(u_j(\alpha), \alpha)$. So each $p_j$ is differentiable with respect to $\alpha$ and we have 
	\begin{align*}
	\frac{d p_j}{d \alpha} &= \frac{ \partial \varphi_j}{\partial u_j} \frac{\partial u_j}{\partial \alpha} + \frac{\partial \varphi_j}{\partial \alpha},\\
	\left| \frac{d p_j}{d \alpha} \right|	&\leq \frac{1}{\cos \phi_j} \frac{C^{(0,1)}_\varphi + C^{(1,1)}_\varphi d_{\min}}{\kappa_{\min} d_{\min}} + \delta_{\xi_j} C^{(0,1)}_\varphi\\
	&\leq \frac{1}{\cos \phi_{\max}} \frac{C^{(0,1)}_\varphi + C^{(1,1)}_\varphi d_{\min}}{\kappa_{\min} d_{\min}} + C^{(0,1)}_\varphi,
	\end{align*}
	where $\delta_i = 1$ if $K_i$ is affected by the deformation and $0$ otherwise.
\end{corollary}
Following the notation in Section \ref{C_f notation}, let 
\begin{align*}
C^{(1)}_u &= \frac{1}{\cos \phi_{\max}} \frac{C^{(0,1)}_\varphi + C^{(1,1)}_\varphi d_{\min}}{\kappa_{\min} d_{\min}},\\ 
C^{(1)}_p &= \frac{1}{\cos \phi_{\max}} \frac{C^{(0,1)}_\varphi + C^{(1,1)}_\varphi d_{\min}}{\kappa_{\min} d_{\min}} + C^{(0,1)}_\varphi.
\end{align*}

\begin{remark} \label{rem multiple obstacles}
	Theorem \ref{thm 6 diff u} assumes that only one obstacle is being deformed. Without this assumption, we have $b_{\max} = \frac{4 C^{(0,1)}_\varphi}{d_{\min}} + 2 C^{(1,1)}_\varphi$ instead of $\frac{2 C^{(0,1)}_\varphi}{d_{\min}} + 2 C^{(1,1)}_\varphi$. So the theorem still holds for deformations of multiple obstacles if we simply replace $\frac{C^{(0,1)}_\varphi + C^{(1,1)}_\varphi d_{\min}}{\kappa_{\min} d_{\min}}$ with $\frac{2C^{(0,1)}_\varphi + C^{(1,1)}_\varphi d_{\min}}{\kappa_{\min} d_{\min}}$.
\end{remark}

\section{Higher derivatives of parameters} 

Let $K(\alpha)$ be a $\mathcal{C}^{r, r'}$ billiard deformation, and let $k \leq \min\{r, r'\}$. The function $h_j = \frac{\partial G}{\partial u_j}$ depends only on $\{\varphi_j\}_j$, so all of its derivatives can be estimated as follows:
$$\left| \frac{\partial^{q'} }{\partial \alpha^{q'}} \nabla^q h_j \right| \leq C_h^{(q, q')},$$
where $C_h^{(q, q')}$ is a constant that depends on the constants $\{C^{(k, k')}_\varphi \}_{k+k' \leq q}$ (but does not depend on $n$).

\begin{lemma} \label{Higher derivatives 2D}
	Let $K(\alpha)$ be a $\mathcal{C}^{r,r'}$ billiard deformation and fix a finite admissible sequence $\xi \in \Sigma_n$. For every $1 \leq q \leq \min\{r - 1, r'\}$, there there exists a constant $C_u^{(q)}$ such that
	$$\left \| \frac{\partial^q u}{\partial \alpha^q} \right \|_{\infty} \leq C_u^{(q)}$$
\end{lemma}

\begin{proof}
We already have the initial case $q = 1$. For a proof by induction, suppose that for some $q < \min\{r - 1, r'\}$, there exist constants $C_u^{(1)}, \ldots C_u^{(q)}$ such that
$$\left \| \frac{\partial^k u}{\partial \alpha^k} \right \|_{\infty} \leq C_u^{(k)},$$
for all $k = 1, \ldots, q$. We show the same is true for $q+1$. We need to take the $q$'th total derivative of $h_j(u(\alpha), \alpha)$ with respect to $\alpha$. There is a formula known as ``Fa\'{a} di Bruno's formula'' \cite{FaadiBruno}, which applies the chain rule for scalar functions an arbitrary number of times. In this case, we require a generalization of Fa\'{a} di Bruno's formula which applies to a scalar function of a vector function of a scalar. First, write $\textbf{x}(\alpha) = (u_0(\alpha), \ldots, u_{n-1}(\alpha), \alpha)$. Then the formula from \cite{FaadiBrunogeneralization} can be applied to $h_j(\textbf{x}(\alpha))$. It is a long formula involving sums over integer partitions, but essentially the total derivative can be written as 
$$
\frac{d^{q+1} h_j}{d \alpha^{q+1}}(\textbf{x}(\alpha))	=  \sum_{j = 0}^{n-1}\frac{\partial^2 G}{\partial u_j \partial u_i}  \frac{\partial^{q+1} u_j}{\partial \alpha^{q+1}} 
+ P \left( \begin{array}{lr}
\left\{ \frac{\partial^{k'} }{\partial \alpha^{k'}} \nabla^k h_j : k +k' \leq q + 1 \right\},\\
\left\{ \frac{\partial^l u_j}{\partial \alpha^l} :  l \leq q \right\}
\end{array} \right),
$$
where $P$ is a polynomial over the elements of the two sets. All of the arguments of $P$ can be estimated by the known constants $C_h^{(k, k')}$ and $C_u^{(l)}$ ($k + k' \leq q, l \leq q$), and the inverse of $H$ is already bounded. So it is possible to calculate a constant $C^{(q+1)}_u$ such that
$$\left\| \frac{\partial^{q+1} u}{\partial \alpha^{q+1}} \right\|_{\infty} \leq C^{(q+1)}_u.$$
So by induction, these estimates exist up to the $\min\{r-1, r'\}$'th derivative. 

\end{proof}

\section{Extension to aperiodic trajectories} \label{6 Extension to aperiodic trajectories}

We now consider trajectories in the non-wandering set that are not periodic. 
Define the \textit{symbol space} for the whole non-wandering set by 
\begin{align*}
\Sigma		&= \{ \xi = (\ldots, \xi_{-1},\xi_0, \xi_1, \ldots)	: \xi_i \in \{1, \ldots m\}, \xi_i \neq \xi_{i+1}\},\\
\Sigma^+	&= \{\xi = (\xi_0, \xi_1, \ldots)	: \xi_i \in \{1, \ldots , m \}, \xi_i \neq \xi_{i+1}\}.
\end{align*}
The two-sided and one-sided subshifts $\sigma: \Sigma \rightarrow \Sigma$ and $\sigma: \Sigma^+ \rightarrow \Sigma^+$ are defined by $(\sigma \xi)_i = \xi_{i+1}$.

The periodic sequences are dense in $\Sigma$, and the periodic points are dense in $M_0$. This follows from  \cite[Lemma 10.2.1]{Stoyanovbook}.

We will use the following proposition about uniformly convergent sequences.
\begin{proposition} \label{differentiability of uniformly convergent sequence}
	Let $m$ be a positive integer, let $I$ be an interval and let $X \subset \mathbb{R}^D$. Then for any sequence of $\mathcal{C}^m$ functions $f_n: I \rightarrow X$, if $f_n$ converges pointwise to $f$ and the $k$'th derivative $f^{(k)}_n$ converges uniformly to a function $g_k$ for all $k \leq m$, then $g$ is differentiable and $f^{(m)} = g_m$.
\end{proposition}
\begin{proof} 
The case $m = 1$ is well known and can be found in \cite{RealAnalysis}, and the rest can be shown by induction. 
\end{proof}
Let $\xi \in \Sigma$, and define a sequence of periodic sequences $\{\xi^{(n)}\}_n$ in $\Sigma$ by $\xi^{(n)}_j = \xi_{(j \text{ mod } n)}$, so that $\xi$ and $\xi^{(n)}$ are on the same $n$-cylinder. 
\begin{corollary} \label{xi^n}
	Note that $\xi^{(n)}$ is equivalent to a string in $\Sigma_n$. Then the following limit exists: 
	$$\chi(\xi) = \lim_{n \rightarrow \infty} \chi \left(\xi^{(n)} \right),$$
	and $\chi: \Sigma \rightarrow M_0$ is the inverse of $\xi: M_0 \rightarrow \Sigma, x \mapsto \xi$. 
	%\begin{proof}
	%We have 
	%$$\|\chi \xi - \chi \xi^{(n)} \| \leq C \delta^n \rightarrow 0,$$
	%so $\chi \xi^{(n)}$ absolutely converges to $\chi \xi$. 
	%\end{proof}
\end{corollary}

For any $\xi \in \Sigma$ and any $j \in \mathbb{Z}$, let $u_j(\xi, \alpha)$ be the parameter such that $p_j(\xi, \alpha) = \varphi_j(u_j(\xi, \alpha), \alpha)$ is the point $\pi B^j \chi \xi$. We show that these aperiodic trajectories satisfy the same derivative estimates as the periodic orbits.

\begin{theorem} \label{thm aperiodic traj}
	Let $K(\alpha)$ be a $\mathcal{C}^{(r, r')}$ billiard deformation with $r \geq 2, r' \geq 2$, and let $\xi \in \Sigma$. Then $u_j(\xi, \alpha)$ is $\mathcal{C}^1$ with respect to $\alpha$ and 
	$$\left| \frac{d u_j(\xi, \alpha)}{d \alpha} \right| \leq C^{(1)}_u$$
	where $C^{(1)}_u$ is defined by (6.4.1). 
\end{theorem}

\begin{proof}
Let $\xi \in \Sigma$ and define a sequence of finite admissible sequences $\{\xi^{(n)}\}_n$ with $\xi^{(n)} \in \Sigma_n$, such that $\xi^{(n)} \rightarrow \xi$. By Corollary \ref{xi^n}, for any fixed $j$ we have 
$$u_j(\xi^{(n)}, \alpha) \rightarrow u_j(\xi, \alpha) \mbox{ as } n \rightarrow \infty.$$
Let $f_n(\alpha) = \frac{d}{d\alpha} u_j(\xi^{(n)}, \alpha)$. This sequence is uniformly bounded for $\alpha \in I$, and so are its derivatives (by Lemma \ref{Higher derivatives 2D}), so it is equicontinuous. So by the Arzel\`{a}-Ascoli theorem \cite{Ascoli}, it has a uniformly (for $\alpha \in I$) convergent subsequence $$f_{n_k}(\alpha) = \frac{d}{d\alpha} u_j(\xi^{(n_k)}, \alpha).$$
Let $f_{n_k}(\alpha) \rightarrow f(\alpha)$ as $k \rightarrow \infty$. By Proposition \ref{differentiability of uniformly convergent sequence}, $u_j(\xi, \alpha)$ is differentiable with respect to $\alpha$, and 
%$$\frac{\partial}{\partial \alpha} f(\alpha) = \lim_{k \rightarrow \infty} \frac{\partial}{\partial \alpha} f_{n_k}(\alpha).$$
$$\frac{d}{d \alpha} u_j(\xi, \alpha) = \lim_{k \rightarrow \infty} \frac{d}{d \alpha} u_j(\xi^{(n_k)}, \alpha).$$
Then by Theorem \ref{thm 6 du bounded}, we have
$$\left| \frac{d u_j(\xi, \alpha)}{d \alpha} \right| \leq C^{(1)}_u.$$
\end{proof}

\begin{corollary} \label{cor aperiodic traj}
	Let $K(\alpha)$ be a $\mathcal{C}^{(r, r')}$ billiard deformation with $r, r' \geq 2$, and let $\xi \in \Sigma$. Then $u_j(\xi, \alpha)$ is $\mathcal{C}^{\min\{r-1, r'-1\}}$ with respect to $\alpha$, and all of its derivatives bounded by the same constants $C^{(q)}_u$ for the periodic trajectories.
\end{corollary}

\begin{proof}
Fix some $j \in \mathbb{Z}$. We prove by induction that for all $1 \leq q \leq \min\{r - 1, r' - 1 \}$, the function $u_j(\xi, \alpha)$ is $\mathcal{C}^q$, and there exist subsequences
$$\{n_{q, k}\}_k \subset \ldots \subset \{n_{2,k}\}_k \subset \{n_k\}_k,$$
such that the following limit is uniform
$$\lim_{k \rightarrow \infty} \frac{d^q}{d \alpha^q} u(\xi^{(n_{q, k})}, \alpha) = \frac{d^q}{d \alpha^q} u(\xi, \alpha).$$
The initial case $q = 1$ is already proven. Suppose it is true for some $q \leq \min\{r - 2, r' - 2\}$. Then the sequence
$$\frac{d^{q+1} u}{d \alpha^{q+1}} (\xi^{(n_{q,k})}, \alpha)$$
is uniformly bounded and equicontinuous for all $k$ (by Lemma \ref{Higher derivatives 2D}), so by the Arzel\`{a}-Ascoli theorem it has a uniformly convergent subsequence 
$$\frac{d^{q+1} u}{d \alpha^{q+1}} (\xi^{(n_{q+1,k})}, \alpha).$$
So by Proposition \ref{differentiability of uniformly convergent sequence}, $u(\xi, \alpha)$ is $\mathcal{C}^{q+1}$ and 
$$\frac{d^{q+1}}{d \alpha^{q+1}} u(\xi, \alpha) = \lim_{n \rightarrow \infty} \frac{d^{q+1}}{d \alpha^{q+1}} u(\xi^{(n)}, \alpha).$$
Furthermore
$$\left| \frac{\partial^{q+1} u_j(\xi, \alpha)}{d \alpha^{q+1}} \right| \leq C^{(q+1)}_u.$$
So by induction, $u$ is at least $\mathcal{C}^{\min\{r - 1, r' - 1\}}$. 
\end{proof}

\begin{remark}
	The only reason $u_j(\xi, \alpha)$ is only $\mathcal{C}^{\min\{r - 1, r' - 1\}}$ and not necessarily $\mathcal{C}^{\min\{r - 1, r'\}}$ is the equicontinuity requirement for the Arzel\`{a}-Ascoli theorem. It may be possible to show that $u_j(\xi, \alpha)$ is $\mathcal{C}^{\min\{r-1, r'\}}$ with another method.
\end{remark}

\begin{corollary}
	For all $\xi \in \Sigma$, the periodic points $p(\xi, \alpha)$ are at least $\mathcal{C}^{\min\{r - 1, r' - 1\}}$ with respect to $\alpha$, and 
	$$\left| \frac{d p(\xi, \alpha)}{d \alpha} \right| \leq C_p.$$
\end{corollary}

\section{Derivatives of other billiard characteristics} \label{6 Derivatives of other properties}

\subsection{Estimating derivatives of distances, curvatures and collision angles} \label{estimating d kappa phi}

From here on, let $K(\alpha)$ be a $\mathcal{C}^{(r, r')}$ billiard deformation with $r \geq 4, r' \geq 2$. 
We can use the upper bound on $\frac{\partial u_j}{\partial \alpha}$ to estimate the derivatives of other characteristics of billiard trajectories, specifically the distances $d_j$, curvature $\kappa_j$ and angles $\phi_j$. We will not estimate the higher derivatives of these functions. Fix a sequence $\xi \in \Sigma$. The distance function $d_j(\alpha)$ is $\mathcal{C}^{\min\{r - 1, r' - 1\}}$ and we have 
\begin{align*}
d_j 									&= |\varphi_j(u_j(\alpha), \alpha) - \varphi_{j-1}(u_{j-1}(\alpha), \alpha)|\\
\frac{\partial d_j}{\partial \alpha}	&\leq \frac{\partial \varphi_j}{\partial u_j} \frac{\partial u_j}{\partial \alpha} +\frac{\partial \varphi_{j-1}}{\partial u_{j-1}} \frac{\partial u_{j-1}}{\partial \alpha} + \delta_j + \delta_{j-1}\\
&\leq 2C^{(1,0)}_\varphi C^{(1)}_u + \delta_{\xi_j} + \delta_{\xi_{j-1}}.
\end{align*}
The derivative of $\kappa_j$ can also be bounded using the following billiard constant. Recall that $\left|\frac{\partial^3 \varphi_j}{\partial u_j^3} \right|$, $\left|\frac{\partial^3 \varphi_j}{\partial u_j^2 \partial \alpha} \right|$ are bounded above by $C^{(3, 0)}_\varphi$ and $C^{(2, 1)}_\varphi$ respectively. Then $\kappa_j$ is $\mathcal{C}^{\min\{r - 3, r'-1\}}$ and
\begin{align*}
\left| \frac{d \kappa_j}{d \alpha} \right|	&= \left|\frac{\partial^3 \varphi_j}{\partial u_j^3} \frac{\partial u_j}{\partial \alpha} + \frac{\partial^3 \varphi_j}{\partial u_j^2 \partial \alpha}\right|\\
&\leq C^{(3, 0)}_\varphi C^{(1)}_u + C^{(2, 1)}_\varphi = C_\kappa.
\end{align*}
The collision angle $\phi_j$ satisfies $\cos 2\phi_j = \frac{(p_{j+1} - p_j) \cdot (p_{j-1} - p_j)}{|p_{j+1} - p_j||p_{j-1} - p_j|}$. Hence, each $\phi_j$ is $\mathcal{C}^{\min\{r - 1, r'-1\}}$ and we have:

\begin{align*}
\left| \frac{d \cos 2\phi_j}{d \alpha} \right|	&= \left| \frac{p_{j+1} - p_j}{|p_{j+1} - p_j|}\cdot \frac{\partial}{\partial \alpha}\frac{p_{j-1} - p_j}{|p_{j-1} - p_j|} - \frac{p_{j-1} - p_j}{|p_{j-1} - p_j|}\cdot \frac{\partial}{\partial \alpha}\frac{p_{j+1} - p_j}{|p_{j+1} - p_j|}\right|\\
&\leq 2\frac{2C^{(1)}_u + \delta_{\xi_{j+1}} + \delta_{\xi_j}}{|p_{j+1} - p_j|} + 2\frac{2C^{(1)}_u + \delta_{\xi_{j-1}} + \delta_{\xi_j}}{|p_{j-1} - p_j|}\\
&\leq \frac{8C^{(1)}_u + 4\delta_{\xi_j} + 2\delta_{\xi_{j+1}} + 2\delta_{\xi_{j-1}}}{d_{\min}},\\
\cos \phi_j 									&= \sqrt{\frac{\cos 2 \phi_j + 1}{2}},\\
\left|\frac{d \cos \phi_j}{d \alpha} \right|	&\leq \frac{4C^{(1)}_u + 2\delta_{\xi_j} + \delta_{\xi_{j+1}} + \delta_{\xi_{j-1}}}{2d_{\min} \cos \phi_j}.
\end{align*}
Denote this upper bound by $C_\phi$ (this is a slight departure from the notation defined in Section \ref{C_f notation}). We will also use the expression $\gamma_j = \frac{2 \kappa_j}{\cos \phi_j}$. This is $\mathcal{C}^{\min\{r-3, r'-1\}}$ and we have
\begin{align*}
\left| \frac{\partial \gamma_j}{\partial \alpha} \right|	&\leq \frac{2 (C^{(3, 0)}_\varphi C^{(1)}_u + C^{(2, 1)}_\varphi)}{\cos \phi_j} + \frac{2 \kappa_j}{\cos \phi_j} \frac{4C^{(1)}_u + 2 \delta_{\xi_j} + \delta_{\xi_{j-1}} + \delta_{\xi_{j+1}}}{2 d_{\min} \cos^2 \phi_j}\\
&\leq \frac{2 (C^{(3, 0)}_\varphi C^{(1)}_u + C^{(2, 1)}_\varphi)}{\cos \phi_{\max}} + \frac{2 \kappa_{\max}}{\cos \phi_{\max}} \frac{2C^{(1)}_u + 1}{d_{\min} \cos^2 \phi_{\max}}.
\end{align*}
Denote this upper bound by $C_\gamma$.

\subsection{Stable and unstable manifolds}

With the no-eclipse condition $(\textbf{H})$, the billiard map $B$ and the flow $S_t$ are examples of an Axiom A diffeomophism and an Axiom A flow respectively. That is, the non-wandering set is hyperbolic, and the periodic points are dense. It is well known (see e.g. \cite{KatokHasselblatt}) that for any point $x \in M_0$ there exist stable and unstable subspaces $E^{(s)}(x)$ and $E^{(u)}(x)$, and local stable and unstable manifolds $W_{\varepsilon}^{(s)}(x), W_{\varepsilon}^{(u)}(x) \subset M$. The stable manifold is simply the time reversal of the unstable manifold, that is $W_{\varepsilon}^{(s)}(x) = \mbox{Refl} W_{\varepsilon}^{(s)} (\mbox{Refl} x)$, where Refl: $\hat{Q} \rightarrow \hat{Q}$ is a bi-Lipschitz involution given by

\begin{displaymath}
\mbox{Refl}(q,v) = \left\{
\begin{array}{lr}
(q, -v) & \mbox{ for } q \in \mbox{int} Q\\
(q, 2 \langle n_K(q), v \rangle n_K(q) - v \rangle), & \mbox{ for } q \in \partial K.
\end{array}
\right.
\end{displaymath} 

\begin{definition} \cite{Pesinbook} \label{conformal}
	An Axiom A diffeomorphism $f$ with a hyperbolic set $\Lambda$ is called $u$-conformal (respectively, $s$-conformal) if there exists a continuous function $a^{(u)}(x)$ (respectively, $a^{(s)}(x)$) on $\Lambda$ such that $d_x f|_{E^{(u)}(x)} = a^{(u)}(x) \mbox{Isom}_x$ for all $x \in \Lambda$ (respectively, $d_x f|_{E^{(s)}(x)} = a^{(s)}(x) \mbox{Isom}_x$ for all $x \in \Lambda$), where $\mbox{Isom}_x$ is an isometry of $E^{(u)}$ or $E^{(s)}$. Then $f$ is called \textit{conformal} if it is both $u$-conformal and $s$-conformal.
\end{definition}

Since the stable and unstable subspaces are each one dimensional, the billiard map in the plane is trivially conformal. For higher dimensional billiards, $B$ is not conformal in general. We now define the convex fronts used to calculate and differentiate the functions $a^{(u)}$ and $a^{(s)}$.

\begin{definition}
	Let $z = (q,v) \in \Omega$ and let $z_0$ be the unique point on $M_0$ such that $S_t z_0 = z$ for some $t \geq 0$. Let $X = X(z) \subset Q$ be the unique convex curve containing $q$ such that for any $x_0 \in \Unstab{z_0}$, there exists $t \geq 0$ and $x \in X$ such that $S_t(x_0) = (x, \nu_X(x))$. Then $X$ is called a \textit{convex front}. Then for any $x \in X(z)$, define $k_z(x)$ to be the curvature of $X(z)$ at $x$. If $z \in M_0$ and $x \in \Unstab{z}$, then 
	$$k_z(x) = \lim_{t \downarrow 0} k_{S_t z}(S_t x).$$
\end{definition}
$X(z)$ is a $\mathcal{C}^r$ curve and the map $z \mapsto X(z)$ is at least $\mathcal{C}^1$ in general \cite{Sinai,ChaoticBilliards}. However if $y \in X(z)$ then the curve $X(y)$ overlaps with $X(z)$, so $z \mapsto X(z)$ is $\mathcal{C}^r$ when restricted to these curves. For a fixed $z$, the map $x \mapsto k_z(x)$ is $\mathcal{C}^{r-2}$ since curvature involves the second derivative.

Recall that the billiard ball map is an Axiom A diffeomorphism, and there exist functions $a^{(s)}, a^{(u)}: M_0 \rightarrow \mathbb{R}$, such that $(d_xB)v = a^{(s)}(x) \mbox{Isom}_x v$ for all $v \in E^{(s)}$ and $(d_xB)v = a^{(u)}(x) \mbox{Isom}_x v$ for all $v \in E^{(u)}$. 

\begin{proposition} 
	Let $K$ be a planar open billiard. Then the billiard map $B$ is conformal on its stable and unstable manifolds, and 
	$$a^{(u)}(x) = 1 + d(x) k_x(x), \hspace{4pt} a^{(s)}(x) = \frac{1}{1 + d(x) k_x(x)}.$$
\end{proposition}
\begin{proof}
See e.g. Lemma 2.1 of \cite{Sinai}, or (3.40) in \cite{ChaoticBilliards}.

\end{proof}

\noindent Note that $|a^{(s)}(x)| < 1$ and $|a^{(u)}(x)| > 1$ for all $x \in M_0$.

\subsection{Curvature of unstable manifolds} \label{6 Curvature of unstable manifolds} \hspace{2pt}

\begin{definition}
	For a fixed $\alpha \in I$ and $\zeta \in \Sigma$, we have a point $z = \chi(\zeta)$ and a convex front $X(z)$. Any point $x \in X$ has the same ``past'' as $z$, in the sense that $B^{-j}z$ and $B^{-j}x$ are on the same obstacle for all $j \geq 0$. For any $\xi \in \Sigma^+$, let $\chi^+_{\alpha, \zeta}(\xi)$ be the unique point $x$ on $X(z)$ satisfying $B^j x \in K_{\xi_j}$ for all $j \geq 0$. 
\end{definition}

Fix a sequence $\zeta \in M_0$ and let $z = \chi(\zeta)$. Then for any point $x \in X(z)$ (sufficiently close to $z$), let $k_j(x) = k_{B^j z}(B^j x)$ be the curvature of the convex front $X(B^j z)$ at $B^j x$. We have the following well-known reccurance relation for $k_j$ (see e.g. \cite{Sinai,ChaoticBilliards}). 
$$k_{j+1} = \frac{k_j}{1 + d_j k_j} + \gamma_j.$$
This equation is smooth for all $d_j, k_j > 0$. Since $\gamma_j$ is $\mathcal{C}^{\min\{r-3, r'-1\}}$ with respect to $\alpha$, $k_j$ is $\mathcal{C}^{\min\{r-3, r'-1\}}$. If $x$ is periodic with period $n$, then $k_n = k_0$ and it is possible to solve these equations for $k_0$. We can bound $k_{\min} \leq k_j \leq k_{\max}$, where $k_{\min}, k_{\max}$ are constants calculated in \cite{Kenny} and \cite{Wright1}.

Now by writing $x = \chi^+_{\alpha,\zeta}(\xi)$ we can differentiate with respect to $\alpha$ to get
\begin{align*}
\frac{d k_{j+1}}{d \alpha}	&= \frac{\frac{d k_j}{d \alpha} (1 + d_j k_j) - k_j (\frac{d d_j}{d \alpha} k_j + d_j \frac{d k_j}{d \alpha})}{(1 + d_j k_j)^2} + \frac{d \gamma_j}{d \alpha}\\
&= \frac{1}{(1 + d_j k_j)^2}  \frac{d k_j}{d \alpha} - \frac{k_j^2}{(1 + d_j k_j)^2} \frac{d d_j}{d \alpha} + \frac{d \gamma_j}{d \alpha}. 													
\end{align*}
Let $\beta_j = \frac{1}{(1 + d_j k_j)^2}$ and $\eta_j = \frac{d \gamma_j}{d \alpha} - \frac{k_j^2}{(1 + d_j k_j)^2} \frac{d d_j}{d \alpha}$. We have $\beta_j \leq \beta_{\max} = \frac{1}{(1 + d_{\min} k_{\min})^2}$, and $\eta_j \leq \eta_{\max} = C_\gamma + \frac{k_{\max}^2(2C^{(1)}_u + 1)}{(1 + d_{\min} k_{\max})^2} $. Then

\begin{align*}
\frac{d k_0}{d \alpha}	&= \frac{d k_n}{d \alpha} = \eta_{n-1} + \beta_{n-1} \frac{d k_{n-1}}{d \alpha}\\
&= \eta_{n-1} + \beta_{n-1}\eta_{n-2} + \ldots + \beta_1 \ldots \beta_{n-1}\eta_0 + \beta_0 \ldots \beta_{n-1}\frac{d k_0}{d \alpha}\\
&= \frac{1}{1 - \beta_0 \ldots \beta_{n-1}} \left(\eta_{n-1} + \beta_{n-1}\eta_{n-2} + \ldots + \beta_1 \ldots \beta_{n-1}\eta_0  \right)\\
&\leq \frac{1}{1 - \beta_{\max}^n} \left(1 + \beta_{\max} + \ldots + \beta_{\max}^{n-1} \right) \max_j \eta_j\\
&\leq \frac{1}{1 - \beta_{\max}}\left(C_\gamma + \frac{k_{\max}^2}{(1 + d_{\min} k_{\max})^2} (2C^{(1)}_u + 1)\right)\\
&\leq C_k.
\end{align*}

\begin{definition}
	Fix a sequence $\zeta \in \Sigma$. Define functions $\psi_{\alpha,\zeta}^{(u)}, \psi_{\alpha,\zeta}^{(s)}: \Sigma^+ \rightarrow \mathbb{R}$ as follows:
	\begin{align*}
	\psi_{\alpha,\zeta}^{(u)}(\xi) &= \log (1 + d(\chi^+_{\alpha,\zeta} \xi) k(\chi^+_{\alpha,\zeta} \xi)), \\
	\psi_{\alpha,\zeta}^{(s)}(\xi) &= -\log (1 + d(\chi^+_{\alpha,\zeta} \xi) k(\chi^+_{\alpha,\zeta} \xi)).
	\end{align*} 
\end{definition}
For a fixed $\xi \in \Sigma^+$, these functions are $\mathcal{C}^{\min\{r - 3, r' - 1\}}$ with respect to $\alpha$, and we have

\begin{align*}
\left| \frac{d \psi_{\alpha,\zeta}^{(u)}}{d \alpha}	\right|&= \frac{\frac{d d(x)}{d\alpha} k(x)+ \frac{d k(x)}{d\alpha} d(x)}{1 + d(x) k(x)}\\
&\leq \frac{C_d k(x) + C_k d(x)}{1 + d(x) k(x)}.
\end{align*}
The expression $\frac{C_d k + C_k d}{1 + dk}$ as a function of $d, k$ reaches its maximum at one of the four corners of the rectangle $[d_{\min}, d_{\max}] \times [k_{\min}, k_{\max}]$. Denote this maximum by $C_\psi$. Using the reflection property, it is easy to see that
$$\left| \frac{d \psi_{\alpha,\zeta}^{(u)}(\xi)}{d \alpha} \right|\leq C_\psi \mbox{ and } \left| \frac{d \psi_{\alpha,\zeta}^{(s)}(\xi)}{d \alpha} \right|\leq C_\psi \mbox{ for all } \xi \in \Sigma^+.$$

\section{Topological pressure and Bowen's equation}

\subsection{Entropy and pressure}

Let $X$ be a compact metric space, $f: X \rightarrow X$ a continuous map, $\Lambda \subset X$ a hyperbolic $f$-invariant subset, and $\psi: X \rightarrow \mathbb{R}$ a continuous function. Denote by $\mathfrak{M}(X)$ the set of all $f$-invariant Borel ergodic measures on $X$. Let $h_\mu(f)$ denote the topological entropy with respect to a measure $\mu \in \mathfrak{M}(X)$, and let $P(\psi) = P_{\Lambda}(\psi)$ denote the topological pressure on $\Lambda$, as defined as in \cite{Pesinbook} or \cite{Walters}. The \textit{variational principle} is

$$P_\Lambda(\psi) = \displaystyle \sup_{\mu \in \mathfrak{M}(X)} \left (h_\mu(f) + \int_\Lambda \psi d \mu \right).$$
There is a unique equilibrium measure $\mu = \mu(\psi)$ corresponding to $\psi$ that satisfies $P_\Lambda(\psi) = h_{\mu}(f) + \int_\Lambda \psi d \mu$ \cite{Pesinbook}.

\begin{proposition}
	For an open billiard, the entropy of the billiard map $B$ is given by $h(B) = \log (m-1)$ where $m$ is the number of obstacles. 
\end{proposition}
\begin{proof}
From \cite{Katok, Bowen}, we have
$$h_{\text{top}} = \lim_{n \rightarrow \infty} \frac{1}{n} T(m, n),$$
where $T(m, n)$ is the number of $n$-periodic trajectories with $m$ obstacles. From \cite{Stoyanov1}, this is equal to $\log (m-1)$. 
\end{proof}

\subsection{Classical topological pressure of a function via separated sets} \label{Pressure of a function by separated sets}
%Let $X$ be a compact metric space, $f: X \rightarrow X$ a continuous map, $Y \subset X$ a hyperbolic $f$-invariant subset, and $\psi: X \rightarrow \mathbb{R}$ a continuous function.

For $\varepsilon > 0, n \in \mathbb{N}$, let
$$P_n(f, \psi, Y, \varepsilon) = \sup \left \{\sum_{x \in E } \exp S_n \psi(x)  : E \text{ is $(n,\varepsilon)-$separated} \right\}.$$
Then the \textit{classical topological pressure} is 
$$P_{\text{classical}}(f, \psi, Y) = \lim_{\varepsilon \rightarrow 0} \limsup_{n \rightarrow \infty} \frac{1}{n} \log P_n(f, \psi, \varepsilon).$$
The \textit{topological entropy} is defined by $h_{\text{top}}(f, Y) = P(f, 0, Y)$. When $f: M \rightarrow M$ has a hyperbolic set $\Lambda$, we will write $h_{\text{top}}(f)$.

\subsection{Pressure on the symbol space} \label{PP pressure}
Pressure and entropy can also be defined using operators on the symbol space $\Sigma^+$ (see \cite{Parry-Pollicott}). For a Lipschitz function $\psi \in C(\Sigma^+)$, that is $\psi:\Sigma^+ \rightarrow \mathbb{R}$, define the Ruelle operator $L_\psi: C(\Sigma^+) \rightarrow C(\Sigma^+)$ by $(L_\psi w)(x) = \displaystyle \sum_{\sigma \xi' = \xi} e^{\psi \xi'} w(\xi')$. Then $L_\psi$ is a bounded linear operator. The Ruelle-Perron-Frobenius theorem guarantees a simple maximum positive eigenvalue $\beta$ for $L_\psi$. We define $P_{\Sigma^+}(\psi) = \log \beta$. The topological entropy can then be defined as $h_{\text{top}}(f) = P(0)$. There is a unique probability measure $\tilde{\mu} = \tilde{\mu}(\psi)$ such that 
$$\int_{\Sigma+} L_\psi v d\tilde{\mu} = \beta \int_{\Sigma^+} v d\tilde{\mu}.$$

\begin{proposition} \label{PP prop} (See e.g. Proposition 4.10 from \cite{Parry-Pollicott})\\
	Let $f, g: \Sigma^+ \rightarrow \mathbb{R}$ and let $\tilde{\mu} = \tilde{\mu}(\psi)$. Then 
	$$ \left. \frac{d}{d s} P(f + s g) \right|_{s = 0} = \int_{\Sigma+} g d\tilde{\mu}.$$
\end{proposition}

\begin{corollary} \label{PP cor}
	Let $I$ be an interval in $\mathbb{R}$. Let $f_s: \Sigma^+ \rightarrow \mathbb{R}$ be a function such that for any fixed $\xi \in \Sigma^+$, $f_s(\xi)$ is $\mathcal{C}^2$ with respect to $s$. For any $s_0 \in I$, let $\mu_0$ be the equilibrium measure for $f_{s_0}$. Then
	
	$$\left. \frac{d}{d s} P(f_s) \right|_{s = s_0} = \int_{\Sigma+} \left. \frac{d f_s}{d s} \right|_{s = s_0}  d\mu_0.$$
	
\end{corollary}

\begin{proof}

For $t \in I$, define a function 
\begin{displaymath}
\gamma(t, \xi) = \left\{
\begin{array}{lr}
\left. \frac{\partial f_t(\xi)}{\partial t} \right|_{t = s_0}		&: t = s_0\\
\frac{f_t(\xi) - f_{s_0}(\xi)}{t - s_0} 								&:  \mbox{otherwise}
\end{array}
\right.
\end{displaymath}
Since $f_t$ is at least $\mathcal{C}^2$ with respect to $t$, $\gamma$ is at least $\mathcal{C}^1$. Define a function by 
$$f_{s, t}(\xi) = f_{s_0}(\xi) + (s - s_0) \gamma(t, \xi).$$
For a fixed $\xi$ this is a $\mathcal{C}^1$ function of two variables. Clearly $f_{s,s}(\xi) = f_s(\xi)$. At $s= s_0$ we have $f_{s_0, t} = f_{s_0}(\xi)$ for all $t \in I$. By Proposition \ref{PP prop}, we have
\begin{align*}
\left. \frac{\partial P(f_{s,t})}{\partial s} \right|_{s = s_0}		&= \int_{\Sigma+} \gamma(t, \xi) d\mu_0,\\
\left. \frac{\partial P(f_{s,t})}{\partial s} \right|_{s = t = s_0} &= \int_{\Sigma+} \gamma(s_0, \xi) d\mu_0 = \int_{\Sigma+} \left. \frac{\partial f_s}{\partial s}(\xi) \right|_{s = s_0}  d\mu_0.
\end{align*}
It remains to show that 
$$\left. \frac{d P(f_s)}{d s} \right|_{s = s_0} = \left. \frac{\partial P(f_{s,t})}{\partial s} \right|_{s = t = s_0}.$$
Define a function $p(s, t) = P(f_{s,t}(\xi))$. Since $P$ is analytic, this is a $\mathcal{C}^1$ function of two variables. Consider $t$ as a function $t = t(s) = s$. Then by the chain rule, we have
$$\left. \frac{d P(f_s)}{d s} \right|_{s = s_0} = \left. \frac{d p}{d s}(s,t(s)) \right|_{s = s_0}	= \frac{\partial p}{\partial s}(s_0, s_0) + \frac{\partial t}{\partial s} \frac{\partial p}{\partial t}(s_0, s_0) .$$
The derivative $\frac{\partial p}{\partial t}(s,t)$ is unknown for $s \neq s_0$, but since $p(s_0, t) = P(f_{s_0,t}) = P(f_{s_0})$ is constant with respect to $t$, we have $\frac{\partial p}{\partial t}(s_0, t) = 0$ for any $t$. So 
\begin{align*}
\frac{d p}{d s}(s_0,s_0)						&= \frac{\partial p}{\partial s}(s_0, s_0) + 0,\\
\left. \frac{d P(f_s)}{d s} \right|_{s = s_0}	&= \left. \frac{\partial P(f_{s,t})}{\partial s} \right|_{s = t = s_0} = \int_{\Sigma+} \left. \frac{\partial f_s}{\partial s} \right|_{s = s_0} d\mu_0,
\end{align*}
as required.
\end{proof}

\subsection{Bowen's equation}

Bowen's equation can refer to any of a number of equations of the form $P(-s \psi) = 0$, where $s$ is a dimension, $P$ is the topological pressure, and $\psi$ is a function related to the dynamical system. The first use of the equation was Bowen's paper on quasi-circles \cite{quasicircles}. Manning and McCluskey used Bowen's equation to calculate the Hausdorff dimension of Smale horseshoes in \cite{Manning-McCluskey}. It is used by Barreira and Pesin in \cite{Barreira1,Pesinbook} to calculate the Hausdorff dimension of hyperbolic sets.

\begin{theorem} \label{Pesin theorem} (Theorem 22.1 of \cite{Pesinbook}).
	Let $\Lambda$ be the non-wandering set for a conformal Axiom A diffeomorphism $f$ on a Riemannian manifold $M$. Let $t^{(u)}$, $t^{(s)}$ be the unique roots of Bowen's equation,
	\begin{equation} \label{Bowen's equation}
	P(-t^{(u)} \log |a^{(u)}(x)|) = 0 \mbox{ and } P(t^{(s)} \log |a^{(s)}(x)|) = 0.
	\end{equation}
	Let $\kappa^{(u)}$ and $\kappa^{(s)}$ be the unique equilibrium measures corresponding to the functions $-t^{(u)} \log |a^{(u)}(x)|$ and $t^{(s)} \log |a^{(s)}(x)|$.
	
	\begin{enumerate}
		\item For any $z \in M_0$ and any open set $U \subset W^{(u)}(z)$ such that $U \cap M_0 \neq \emptyset$, 
		$$\dim_H(U \cap M_0) = \underline{\dim}_B(U \cap M_0) = \overline{\dim}_B(U \cap M_0) = \mathcal{D}^{(u)}.$$
		\item For any $z \in M_0$ and any open set $S \subset W^{(s)}(z)$ such that $S \cap M_0 \neq \emptyset$, 
		$$\dim_H(S \cap M_0) = \underline{\dim}_B(S \cap M_0) = \overline{\dim}_B(S \cap M_0) = \mathcal{D}^{(s)}.$$
		\item The dimensions of the hyperbolic set $\Lambda$ are given by
		$$\dim_H(M_0) = \underline{\dim}_B(M_0) = \overline{\dim}_B(M_0) = \mathcal{D}^{(u)} + \mathcal{D}^{(s)}.$$
		\item The numbers $\mathcal{D}^{(u)}$ and $\mathcal{D}^{(s)}$ satisfy
		\begin{equation}
		\mathcal{D}^{(u)} = \frac{h_{\kappa^{(u)}}(f)}{\int_{\Lambda} \log|a^{(u)}(x)| d\kappa^{(u)}} \mbox{ and } \mathcal{D}^{(s)} = \frac{h_{\kappa^{(s)}}(f)}{\int_{\Lambda} \log|a^{(s)}(x)| d\kappa^{(s)}},
		\end{equation}
		where $h_{\mu}(f)$ is the entropy of $f$ with respect to $\mu$.
	\end{enumerate}

\end{theorem} 

\section{Bounds on Hausdorff dimension}
%For this subsection, we fix $z \in M_0$ and $\alpha \in I$. 
Let $\mathcal{D}^{(u)} = \mathcal{D}^{(u)}(z), \mathcal{D}^{(s)} = \mathcal{D}^{(u)}(z)$ be the solution to Bowen's equations
$$P(-\mathcal{D}^{(u)} \log |a^{(u)}|) = 0 \mbox{ and } P(\mathcal{D}^{(s)} \log |a^{(s)}|) = 0.$$
Let $\kappa^{(u)}$, $\kappa^{(s)}$ be the unique equilibrium measures corresponding to the functions $-\mathcal{D}^{(u)} \log |a^{(u)}|$, $\mathcal{D}^{(s)} \log |a^{(s)}|$. Now recall Theorem \ref{Pesin theorem}. 

\begin{theorem}
	Let $\mu_0$ be the measure such that $h_{\emph{top}} = h_{\mu_0}$. Then the dimension of the unstable manifold satisfies
	$$\frac{\log(m - 1)}{\int_{M_0} \log |a^{(u)}| d \mu_0} \leq \mathcal{D}^{(u)} = \frac{h_{\kappa^{(u)}}}{\int_{M_0} \log |a^{(u)}| d\kappa^{(u)}} \leq  \frac{\log(m - 1)}{\int_{M_0} \log |a^{(u)}| d\kappa^{(u)}}.$$
\end{theorem}

\begin{proof}
The topological entropy is 
$$\displaystyle\sup_\mu h_\mu = h_{\text{top}} = \log(m-1).$$  
Since $\kappa^{(u)}$ is the equilibrium measure for $-\mathcal{D}^{(u)} \log |a^{(u)}|$, it satisfies
$$h_{\kappa^{(u)}} - \mathcal{D}^{(u)} \int_{M_0} \log |a^{(u)}| d\kappa^{(u)} = \sup_\mu \left(h_{\mu} - \mathcal{D}^{(u)} \int_{M_0} \log |a^{(u)}| d \mu \right).$$ 
So in particular, 

\begin{align*}
h_{\kappa^{(u)}} - \mathcal{D}^{(u)} \int_{M_0} \log |a^{(u)}| d\kappa^{(u)}	&\geq h_{\mu_0} - \mathcal{D}^{(u)} \int_{M_0} \log |a^{(u)}| d \mu_0\\
0																		&\geq h_{\text{top}} - \mathcal{D}^{(u)} \int_{M_0} \log |a^{(u)}| d \mu_0\\
\mathcal{D}^{(u)}																	&\geq \frac{\log(m - 1)}{\int_{M_0} \log |a^{(u)}| d \mu_0}.
\end{align*}

The other inequality is trivial since $h_{\kappa^{(u)}} \leq h_{\text{top}} = \log(m-1)$.
\end{proof}

\begin{theorem}
	Let $\mu_0$ be the measure such that $h_{\emph{top}} = h_{\mu_0}$. Then the dimension of the unstable manifold satisfies
	$$\frac{\log(m - 1)}{\int_{M_0} \log |a^{(s)}| d \mu_0} \leq \mathcal{D}^{(s)} = \frac{h_{\kappa^{(s)}}}{\int_{M_0} \log |a^{(s)}| d\kappa^{(s)}} \leq  \frac{\log(m - 1)}{\int_{M_0} \log |a^{(s)}| d\kappa^{(s)}}.$$
\end{theorem} 
\begin{proof}
The proof is very similar to the previous theorem.
\end{proof}

In \cite{Kenny} and \cite{Wright1}, estimates are found for the Hausdorff dimension using different methods. The latter gives stronger estimates in some cases and applies to higher dimensions. 
\begin{corollary}
	The dimension of the non-wandering set are effectively the same as the estimates given in \cite{Wright1}.
\end{corollary}

\begin{proof}
%Since the measures $\kappa^{(u)}$ and $\mu_0$ are both ergodic, we have 
%$$\int_{M_0} d \kappa^{(u)} = \int_{M_0} d \mu_0 = 1.$$
Using the estimates for $d$ and $k$ in \cite{Wright1}, we have 
\begin{align*}
\log(1 + d_{\min} k_{\min})				&\leq \int_{M_0} \log (1 + d(x) k(x)) d\kappa^{(u)} = \int_{M_0} \log |a^{(u)}| d\kappa^{(u)},\\
\int_{M_0} \log |a^{(u)}| d\mu_0	&= \int_{M_0} \log (1 + d(x) k(x)) d\mu_0 \leq \log(1 + d_{\max} k_{\max}).
\end{align*}
The same holds for the stable manifolds, so we have

$$\frac{2\log(m-1)}{\log(1 + d_{\max} k_{\max})} \leq \mathcal{D}^{(u)} + \mathcal{D}^{(s)} = \dim_H M_0 \leq \frac{2\log(m-1)}{\log(1 + d_{\min} k_{\min})}$$
which is the estimate in \cite{Wright1} (Theorem 2.1 (i)).
\end{proof}

\section{Derivative of Hausdorff dimension} \label{6 Derivative of dim}

In this section we show that the Hausdorff dimension of the non-wandering set is differentiable with respect to $\alpha$ and that its derivative is bounded by a constant depending only on the deformation. We use the definition of topological pressure $P_{\Sigma^+}$ based on the Ruelle operator from Section \ref{PP pressure}. Recall that for any function $\psi: M_0 \rightarrow \mathbb{R}$, this is related to the classical topological pressure by 
$$P_{\Sigma^+}(\psi \circ \chi^+_{\alpha,\zeta}) = P_{\text{classical}}(B, \psi, M_0).$$
We can rewrite Bowen's equation using this definition of pressure. The Hausdorff dimensions $\mathcal{D}^{(u)}$ and $\mathcal{D}^{(s)}$ are given by
$$P_{\Sigma^+}(\mathcal{D}^{(u)} \psi_{\alpha, \zeta}^{(u)}) = P_{\Sigma^+}(\mathcal{D}^{(s)} \psi_{\alpha, \zeta}^{(s)}) = 0$$
We will focus on the unstable manifolds first.

\begin{theorem}
	Let $K(\alpha)$ be a $\mathcal{C}^{(r, r')}$ billiard deformation with $r \geq 4, r' \geq 2$. Then the Hausdorff dimension $\mathcal{D}^{(u)}$ is at least $\mathcal{C}^{\min\{r-3, r'-1\}}$ with respect to $\alpha$.
\end{theorem}

\begin{proof}
We will use the implicit function theorem \cite{IFT} for the following function of two variables:
$$Q(D, \alpha) = P_{\Sigma^+}(D \psi^{(u)}_{\alpha,\zeta}).$$
Choose an arbitrary $\alpha_0 \in I$, and let $D_0 \in \mathbb{R}^+$ be such that $P_{\Sigma^+}(D_0 \psi^{(u)}_{\alpha_0,\zeta}) = 0$. Note that $D_0 \neq 0$ because the entropy is nonzero. Let $\tilde{\kappa}^{(u)}_0$ be the unique equilibrium measure corresponding to $- D_0 \psi_{\alpha_0,\zeta}^{(u)}$. First we show that 
$$\left. \frac{\partial}{\partial D} Q(D, \alpha) \right|_{D_0, \alpha_0} \neq 0.$$
By Proposition \ref{PP prop}, we have

$$\left. \frac{\partial}{\partial D} P_{\Sigma^+} \left(D_0 \psi^{(u)}_{\alpha_0,\zeta} + (D - D_0) \psi^{(u)}_{\alpha_0,\zeta} \right) \right|_{D - D_0 = 0} = \int_{\Sigma^+} \psi^{(u)}_{\alpha_0,\zeta} d \tilde{\kappa}^{(u)}_0.$$
By Bowen's equation and the variational principle we have
$$0 = P(D_0 \psi^{(u)}_{\alpha_0,\zeta}) = h_{\kappa^{(u)}_0} + \int_{\Sigma^+} D_0 \psi^{(u)}_{\alpha_0,\zeta} d\tilde{\kappa}^{(u)}_0.$$
Rearranging this we get 
$$\left. \frac{\partial}{\partial D} Q(D, \alpha) \right|_{D_0, \alpha_0} = \int_{\Sigma^+} \psi^{(u)}_{\alpha_0,\zeta} d \tilde{\kappa}^{(u)}_0 = -\frac{h_{\kappa^{(u)}_0}}{D_0}.$$
Since the entropy is never zero in our model, this is non-zero. Since $P$ is analytic, the map $\alpha \mapsto P(D \psi^{(u)}_{\alpha, \zeta})$ is $\mathcal{C}^{\min\{r-3, r'-1\}}$ and the map $D \mapsto P(D \psi^{(u)}_{\alpha, \zeta})$ is analytic. So the implicit function theorem (and its corollary) is applicable. There exists a $\mathcal{C}^{\min\{r - 3, r' - 1\}}$ function $\mathcal{D}^{(u)}(\alpha)$, such that $P(\mathcal{D}^{(u)}(\alpha) \psi^{(u)}_{\alpha, \zeta}) = 0$ for all $\alpha \in I$, and furthermore

\begin{equation} \label{eq: IFT pressure}
\frac{\partial Q}{\partial \alpha}(\mathcal{D}^{(u)}(\alpha), \alpha) + \frac{\partial \mathcal{D}^{(u)}}{\partial \alpha} \frac{\partial Q}{\partial D} \left(\mathcal{D}^{(u)}(\alpha), \alpha \right) = 0.
\end{equation}
Since $P(\mathcal{D}^{(u)}(\alpha) \psi^{(u)}_{\alpha, \zeta}) = 0$, the function $\mathcal{D}^{(u)}(\alpha)$ is precisely the Hausdorff dimension of $M_0 \cap U$ for any $U \subset W^{(u)}(z)$ with $M_0 \cap U \neq \emptyset$, for any $z = \chi \zeta \in M_0$ and any $\alpha \in I$. The dimension $\mathcal{D}^{(u)}$ does not depend on the choice of $\zeta$.
\end{proof}

\begin{theorem}
	Let $K(\alpha)$ be a $\mathcal{C}^{(r, r')}$ billiard deformation with $r \geq 4, r' \geq 3$. Then $\frac{\partial D^{(u)}}{\partial \alpha}$ is bounded.
\end{theorem}
\begin{proof}
Since $r' \geq 3$, $\psi^{(u)}_{\alpha, \zeta}$ is at least $\mathcal{C}^2$. So by Corollary \ref{PP cor}, we have 

$$\left. \frac{\partial }{\partial \alpha} P(D_0 \psi^{(u)}_{\alpha, \zeta}) \right|_{\alpha = \alpha_0} = D_0 \int_{\Sigma^+} \left. \frac{\partial \psi^{(u)}_{\alpha, \zeta}}{\partial \alpha} \right|_{\alpha = \alpha_0} d \tilde{\kappa}^{(u)}_0.$$
Then (\ref{eq: IFT pressure}) becomes

$$0 = \left. \frac{\partial Q}{\partial \alpha} + \frac{\partial \mathcal{D}^{(u)}}{\partial \alpha} \frac{\partial Q}{\partial D} \right|_{\alpha = \alpha_0}  = \left. \mathcal{D}^{(u)} \int_{\Sigma^+} \frac{\partial \psi^{(u)}_{\alpha, \zeta}}{\partial \alpha} d \tilde{\kappa}^{(u)}_0 - \frac{\partial \mathcal{D}^{(u)}}{\partial \alpha} \frac{h_{\kappa^{(u)}_0}}{\mathcal{D}^{(u)}}  \right|_{\alpha = \alpha_0} .$$
Since $\alpha_0$ was chosen arbitrarily, for any $\alpha \in I$ we have
\begin{equation}
\frac{\partial \mathcal{D}^{(u)}}{\partial \alpha}	= \frac{\mathcal{D}^{(u)}(\alpha) ^2}{h_{\kappa^{(u)}}} \int_{\Sigma^+} \frac{\partial \psi^{(u)}_{\alpha, \zeta}}{\partial \alpha} d \tilde{\kappa}^{(u)}.
\end{equation}
The integrand is bounded by $C_\psi$, and $\tilde{\kappa}^{(u)}$ is a probability measure, so we have
$$\left| \frac{\partial \mathcal{D}^{(u)}}{\partial \alpha} \right| \leq \frac{\mathcal{D}^{(u)}(\alpha)^2 }{h_{\kappa^{(u)}}} C_\psi.$$
\end{proof}

\begin{theorem}
	Let $K(\alpha)$ be a $\mathcal{C}^{r, r'}$ billiard deformation with $r \geq 4, r' \geq 3$.
	For any sequence $\zeta$, the Hausdorff dimension $D^{(s)}$ of $M_0 \cap W^{(s)}(\chi z)$ is at least $\mathcal{C}^{\min\{r-3, r'-1 \}}$ with respect to $\alpha$, and its first derivative is bounded by $$\frac{D^{(s)}(\alpha)^2 C_\psi}{h_{\kappa^{(s)}}}.$$
\end{theorem}
\begin{proof}
By the reflection property $W^{(u)}_{\varepsilon}(z) = \mbox{Refl} W_{\varepsilon}^{(s)}(\mbox{Refl}(z))$, we have 
$$\psi_{\alpha, \zeta}^{(s)} = - \psi_{\alpha, \zeta}^{(u)} \mbox{ and } \mathcal{D}^{(s)} = \mathcal{D}^{(u)}.$$
So the proof is very similar to the previous theorem. 
\end{proof}

\begin{theorem}
	Let $K(\alpha)$ be a $\mathcal{C}^{r, r'}$ billiard deformation with $r \geq 4, r' \geq 3$. The Hausdorff dimension $\mathcal{D}(\alpha) = \mathcal{D}^{(u)} + \mathcal{D}^{(s)}$ of the non-wandering set is $\mathcal{C}^{\min\{r-3, r'-1 \}}$ with respect to $\alpha$, and its derivative is bounded by the following.
	$$\left| \frac{d \mathcal{D}}{d \alpha} \right| \leq \frac{C_\psi \mathcal{D}}{\log (1 + d_{\min} k_{\min})}.$$
\end{theorem}
\begin{proof}
Since $B$ is conformal, the dimension satisfies $\mathcal{D}(\alpha) = \mathcal{D}^{(s)}(\alpha) + \mathcal{D}^{(u)}(\alpha)$. So
\begin{align*}
\left| \frac{d \mathcal{D}}{d \alpha} \right|	&\leq \frac{\mathcal{D}^{(s)}(\alpha)^2 C_\psi}{h_{\kappa^{(s)}}} + \frac{\mathcal{D}^{(u)}(\alpha)^2 C_\psi}{h_{\kappa^{(u)}}}\\
&\leq \frac{C_\psi \mathcal{D}^{(s)}(\alpha)}{\log(1 + d_{\min} k_{\min})} + \frac{C_\psi \mathcal{D}^{(u)}(\alpha)}{\log(1 + d_{\min} k_{\min})}\\
&\leq \frac{C_\psi \mathcal{D}(\alpha)}{\log(1 + d_{\min} k_{\min})}.
\end{align*}

This follows from the final part of Theorem \ref{Pesin theorem}.
\end{proof}

\begin{remark}
	Note that $r' \geq 3$ is only required for estimating the first derivative. To find the differentiability class of the Hausdorff dimension only $r' \geq 2$ is required.
\end{remark}

\begin{corollary}
	Let $K(\alpha)$ be a real analytic billiard deformation (i.e. real analytic in both $\alpha$ and the parameter $u$). Then $\mathcal{D}(\alpha)$ is a real analytic function of $\alpha$.
\end{corollary}

\begin{proof}
Let $K(\alpha)$ be a real analytic billiard deformation. Then clearly the functions $F$, $G$ and $h$ are all real analytic. By the implicit function theorem for analytic functions \cite{AnalyticIFT}, $u_j(\xi, \alpha)$ is real analytic for any periodic admissible sequence $\xi$. To extend this to aperiodic sequences, we show that $u_j(\xi, \alpha)$ is equal to its Taylor series at any point $\alpha_0$ and for any $\xi \in \Sigma$. Recall from the proof of Corollary 6.6.2 that for each $q$ there exists a sequence $(\xi^{(n_{q,k})})_k$ such that $\displaystyle \lim_{k \rightarrow \infty} \xi^{(n_{q,k})} = \xi$, and
$$\frac{d^q}{d \alpha^q}  u_j(\xi, \alpha) = \lim_{k \rightarrow \infty} \frac{d^q}{d \alpha^q} u_j(\xi^{(n_{q,k})}, \alpha).$$
Now using the definition of real analyticity, we have
\begin{align*}
u_j(\xi, \alpha)	&= \lim_{k \rightarrow \infty} u_j(\xi^{(n_{q,k})}, \alpha)\\
&= \lim_{k \rightarrow \infty} \sum_{q=0}^\infty \frac{(\alpha - \alpha_0)^q}{q!} \left( \left. \frac{d^q}{d \alpha^q} u_j(\xi^{(n_{q,k})}, \alpha) \right|_{\alpha = \alpha_0} \right)\\
&= \sum_{q=0}^\infty \frac{(\alpha - \alpha_0)^q}{q!} \left( \left. \lim_{k \rightarrow \infty} \frac{d^q}{d \alpha^q}  u_j(\xi^{(n_{q,k})}, \alpha) \right|_{\alpha = \alpha_0} \right)\\
&= \sum_{q=0}^\infty \frac{(\alpha - \alpha_0)^q}{q!} \left( \left. \frac{d^q}{d \alpha^q} u_j(\xi, \alpha) \right|_{\alpha = \alpha_0} \right).
\end{align*}
Therefore $u_j(\xi, \alpha)$ is real analytic in $\alpha$. Now the quantities $p, d, \kappa, \phi$ and $\gamma$ are all easily shown to be real analytic. The recurrance relation for $k_j$ can be used to show that $k_j$ is real analytic, and therefore $\psi_{\alpha, \zeta}(\xi)$ is real analytic. Since the pressure is a real analytic operator, it follows that $\mathcal{D}(\alpha)$ is real analytic. 
\end{proof}

%For acknowledgements section, please don't number the section, please begin it with \section*{Acknowledgements}
\section*{Acknowledgment} This work derives from the author's PhD thesis, ``Dimensional Characteristics of the Non-wandering Sets of Open Billiards''. The author wishes to thank L. Stoyanov for many helpful comments and suggestions.

% You may incorporate your references as follows in your main tex file.
% Using BibTex is not recommended but can be handled.

%\medskip
%% The data information below will be filled by AIMS editorial staff
%Received xxxx 20xx; revised xxxx 20xx.
%\medskip

\end{document}